\documentclass[12pt]{amsart}

\usepackage{amsmath,amssymb,mathrsfs}
\usepackage{a4wide}
\usepackage{amsthm} % symboles math, ...

\usepackage{srcltx} %DVIWIN Inverse Search

 \usepackage[dvips]{graphicx,color}
\usepackage{amsopn,cite,mathrsfs}
\usepackage{amsfonts}
\usepackage{eucal}
\usepackage{times}

\numberwithin{equation}{section}

\newtheorem{thm}{Theorem}[section]
\newtheorem{lemma}[thm]{Lemma}

\newcommand{\C}{{\mathbb C}}

\newcommand{\cn}{{\C^n}}

\newcommand{\Tr}{\operatorname{Tr}}

\newcommand{\mo}{\operatorname{MO}}
\newcommand{\bmo}{\operatorname{BMO}}
\newcommand{\bmoa}{\operatorname{BMOA}}
\newcommand{\vmoa}{\operatorname{VMOA}}

\def\A{\ensuremath{{\mathcal{A}}^2(\mu)}}

\def\cn{\ensuremath{\mathbb{C}^n}}

\def\cS{\mathcal S}

\def\cT{\ensuremath{\mathcal{T}}}
\def\Oms{\cn}

\def\littlebloch{\ensuremath{\mathcal{B}_{0}(\Psi)}}

\def\C{\ensuremath{\mathbb{C}}}

\def\A{\ensuremath{{\mathcal{A}}^2(\Psi)}}
\def\cn{\ensuremath{\mathbb{C}^n}}
\def\M{\ensuremath{{\mathcal {B}}(\Psi)}}
\def\Mmp{\ensuremath{{\mathcal{B}_{m}^{p}(\Psi)}}}
\def\Mdp{\ensuremath{{\mathcal{B}_{d}^{p}(\Psi)}}}

\begin{document}

\title[Hankel operators on Fock spaces and related Bergman kernel estimates]
{Hankel operators on Fock spaces \\ and related Bergman kernel
estimates}

\author{Kristian Seip and El Hassan Youssfi}

\address{Seip: Department of Mathematical Sciences\\
         Norwegian University of Science and Technology (NTNU)\\
         NO-7491 Trondheim\\
         Norway}
\email{seip@math.ntnu.no}

\address{Youssfi: LATP, U.M.R. C.N.R.S. 6632, CMI\\
         Universit\'e de Provence\\
         39 Rue F-Joliot-Curie\\
         13453 Marseille Cedex 13, France}
\email{youssfi@gyptis.univ-mrs.fr}



\thanks{The first author is supported by the Research Council of
Norway grant 185359/V30. The second author is supported by the
French ANR DYNOP, Blanc07-198398.}

\begin{abstract} Hankel operators with anti-holomorphic symbols are studied
for a large class of weighted Fock spaces on $\cn$. The weights
defining these Hilbert spaces are radial and subject to a mild
smoothness condition. In addition, it is assumed that the weights
decay at least as fast as the classical Gaussian weight. The main
result of the paper says that a Hankel operator on such a Fock space
is bounded if and only if the symbol belongs to a certain $\bmoa$
space, defined via the Berezin transform. The latter space coincides
with a corresponding Bloch space which is defined by means of the
Bergman metric. This characterization of boundedness relies on
certain precise estimates for the Bergman kernel and the Bergman
metric. Characterizations of compact Hankel operators and Schatten
class Hankel operators are also given. In the latter case, results
on Carleson measures and Toeplitz operators along with
H\"{o}rmander's $L^2$ estimates for the $\overline{\partial}$
operator are key ingredients in the proof.
\end{abstract}

\maketitle

\section{Introduction}
This paper presents the basics of Hankel operators with
anti-holomorphic symbols for a large class of weighted Fock spaces.
Thus certain natural analogues of $\bmoa$, the Bloch space, the
little Bloch space, and the Besov spaces are identified and shown to
play similar roles as their classical counterparts do. We will see
that these spaces contain all holomorphic polynomials and are
infinite-dimensional whenever the weight decays so fast that there
exist functions of infinite order belonging to the Fock space.

The setting is the following. Consider a $C^3$-function $\Psi : [0,
+\infty[ \to [0, +\infty[ $ such that
\begin{equation}\label{admissible}
\Psi'(x)>0, \ \ \ \Psi''(x)\ge 0, \ \   \text{and} \ \ \Psi'''(x)\ge
0.\end{equation} We will refer to such a function as a logarithmic
growth function. Note that \eqref{admissible} effectively says that
$\Psi$ should grow at least as a linear function. Set
$$d\mu_\Psi (z) := e^{-\Psi(|z|^2)} d V(z),$$ where $dV$ denotes
Lebesgue measure on $\cn$, and let ${\mathcal A}^2(\Psi)$ be the
Fock space defined as the closure of the set of holomorphic
polynomials in $L^2( \mu_\Psi).$  We observe that ${\mathcal
A}^2(\Psi)$ coincides with the classical Fock space when $\Psi$ is a
suitably normalized linear function.

It is immediate that
$$s_d := \int_{0}^{+\infty} x^d e^{-\Psi(x)}d x < +\infty $$
for all nonnegative integers $d$. Moreover, as shown in \cite{BY1},
the series
 $$ F_s (\zeta) := \sum_{d=0}^{+\infty} \frac{\zeta^d}{s_{d}}, \ \zeta \in \C,$$
has an infinite radius of convergence and
 ${\mathcal A}^2(\Psi)$ is a reproducing kernel Hilbert space with reproducing kernel
 $$K_\Psi(z, w) = \frac{1}{(n-1)!} F^{(n-1)}_s(\langle z, w\rangle),  \ \ z, w \in \cn.$$
This implies that the orthogonal projection $P_\Psi$ from
$L^2(\mu_\Psi)$
 onto $ {\mathcal A}^2(\Psi)$ can be expressed as
$$(P_\Psi g)(z) = \int_{\cn} K_\Psi(z, w) g(w) d\mu_\Psi(w), \ z \in \cn,$$
for every function $g$ in $L^2(\mu_\Psi)$. The domain of this
integral operator can be extended to include functions $g$ that
satisfy $K_\Psi(z, \cdot) g \in L^1(\mu_\Psi)$ for every $z$ in
$\cn.$ This extension allows us to define (big) Hankel operators.
To do so, denote by $\cT(\Psi) $ the class of all $f$ in
$L^2(\mu_\Psi) $
 such that $f \varphi K_\Psi(z, \cdot) \in L^1(\mu_\Psi)$  for all holomorphic
polynomials $\varphi$ and $z$ in $\cn$ and the function
$$
H_{f}(\varphi) (z) :=  \int_{\cn}K_\Psi(z,w)\varphi(w) \left[
f (z) -
  f  (w)\right] d\mu_\Psi(w), \ \  z\in \cn,$$
is  in $L^2(\mu_\Psi)$. This is a densely defined
operator from  ${\mathcal A}^2(\Psi)  $ into  $L^2(\mu_\Psi)$ which will be called the
Hankel operator  $H_{f}$  with symbol $ f$.
It can be written in the form
$$
H_{f}(\varphi)=  (I-P_\Psi)( f \varphi)
$$
 for all holomorphic polynomials $\varphi$.
It is clear that the class $\cT(\Psi) $ contains all holomorphic
polynomials.

Our main theorem involves the analogues in our setting of the space
$\bmoa$ and the Bloch space. The analogue of $\bmoa$ is most
conveniently defined via the Berezin transform, which for a linear
operator $T$ on ${\mathcal A}^2(\Psi)$ is the function
$\widetilde{T}$ defined on $\Oms$ by
$$\widetilde{T}(z):=
\dfrac{\langle T K_\Psi(\cdot,z),K_\Psi(\cdot,z)
\rangle}{K_\Psi(z,z)}.$$ If $T = M_f$ is the operator of
multiplication by the function $f$, then we just set
$\widetilde{M_f} = \tilde f.$ We set
$$\| f \|_{\bmo} := \sup_{z\in \cn} (\mo f)(z),$$
where $$(\mo f)(z):= \sqrt{\widetilde{|f|^2}(z) - | \tilde
f(z)|^2}$$ and define $\bmo(\Psi)$ as the set of functions $f$ on
$\cn$ for which $\widetilde{|f|^2}(z)$ is finite for every $z$ and
$\|f\|_{\bmo}<\infty$. It is plain that $\bmo(\Psi)$ is a subset of
$\cT(\Psi)$. The space $\bmoa(\Psi)$ is the subspace of $\bmo(\Psi)$
consisting of analytic elements; this space is in turn a subset of
$\cT(\Psi)\cap {\mathcal A}^2(\Psi)$.

We next introduce the Bergman metric associated with $\Psi.$ To this
end, set $\Lambda_{\Psi}(z)=\log K_{\Psi}(z,z)$ and
\[
\beta^2(z, \xi):= \sum_{j, k =1}^n\frac{\partial^2
\Lambda_{\Psi}(z)} {\partial z_j
\partial \bar z_k}  \xi_j
\bar \xi_k\] for arbitrary vectors $z=(z_1,...,z_n)$ and
$\xi=(\xi_1,...,\xi_n)$ in $\cn$. The corresponding distance
$\varrho$ is given by
\begin{equation}\label{bergmanmetric}\varrho(z, w) := \inf_\gamma
\int_{0}^{1} \beta(\gamma(t), \gamma'(t)) dt,\end{equation} where
the infimum is taken over all piecewise $C^1$-smooth curves $\gamma:
[0, 1] \to \cn$ such that $\gamma(0) = z $ and $\gamma(1) = w $. We
define the Bloch space $\M$ to be the space of all entire funtions
$f$ such that \begin{equation}\label{blochdef}\|f\|_{\M}:=
\sup_{z\in \cn} \left[\sup_{\xi \in \cn\setminus\{0\} } \frac{\left
| \langle (\nabla f)(z), \overline{\xi} \rangle\right |}{\beta(z,
\xi)} \right] < + \infty.\end{equation}

In what follows, the function
\[\Phi(x):=x\Psi'(x)\]
will play a central role.  By \eqref{admissible}, we have that both
$\Phi'(x)>0$ and $\Phi''(x)>0$, and it may be checked that
$\Phi'(|z|^2)$ coincides with the Laplacian of $\Psi(|z|^2)$ when
$n=1$ and in general is bounded below and above by positive
constants times this Laplacian for arbitrary $n>1$.

We are now prepared to state our main result.
\newtheorem*{thA}{\bf Theorem A}
\begin{thA}  Let $\Psi$ be a logarithmic growth function, and
suppose that there exists a real number $\eta<1/2$
such that
 \begin{equation}\label{basic}
\Phi''(t)=O(t^{-\frac{1}{2}}[\Phi'(t)]^{1+\eta}) \ \ \ \text{when} \
t\to\infty.
\end{equation}
If  $f$ is an entire function on $\cn$, then the following
statements are equivalent:
\begin{itemize}
\item[(i)]The function $f$ belongs to $\cT(\Psi)$ and the Hankel operator $H_{\bar f} $ on $\A$
is bounded; % from ${\mathcal A}^2(\Psi)$ into $L^2( \mu_\Psi)$;
\item[(ii)] The function $f$ belongs to $\bmoa(\Psi)$;
\item[(iii)] The function $f$ belongs to $\M$.
\end{itemize}
\end{thA}

Note that the additional assumption \eqref{basic} is just a mild
smoothness condition, which holds whenever $\Psi$ is a nontrivial
polynomial or a reasonably well-behaved function of super-polynomial
growth.

As part of the proof of Theorem~A, we will perform a precise
computation of the asymptotic behavior of $\beta(z,\xi)$ when
$|z|\to\infty$. We state this result as a separate theorem.

\newtheorem*{thB}{\bf Theorem B}
\begin{thB}  Let $\Psi$ be a logarithmic growth function, and
suppose that there exists a real number $\eta<1/2$ such that
\eqref{basic} holds. Then we have, uniformly in $\xi$, that
\[ \beta^2(z,\xi)=(1+o(1))\left(|\xi|^2\Psi'(|z|^2)+|\langle
z,\xi\rangle|^2 \Psi''(|z|^2)\right) \ \ \ \text{when} \
|z|\to\infty.\]
\end{thB}

We observe that for the classical Fock space ($\Psi$  a linear
function) we have $\Psi''(x)\equiv 0$, and so the ``directional''
term in $\beta(z,\xi)$ is not present. Note also that $\M$ contains
all polynomials and is infinite-dimensional whenever the growth of
$\Psi'(x)$ is super-polynomial. In the language of entire functions,
this means that ${\mathcal A}^2(\Psi) $ contains functions of
infinite order. When $n=1$, $\beta^2(z,\xi)$ can be replaced by
$\Phi'(|z|^2)|\xi|^2$. The same is also true when $\Psi$ is a
polynomial, because then $\Psi'$ and $\Phi'$ have the same
asymptotic behavior. In the latter case, our two theorems give the
following precise result: If $\Psi$ is a polynomial of degree $d$,
then $\M$ consists of all holomorphic polynomials of degree at most
$d$, cf. Theorem~A in \cite{BY1}.

The implication (i) $\Rightarrow$ (ii) in Theorem~A is standard; it
follows from general arguments for reproducing kernels. Likewise,
the implication (ii) $\Rightarrow$ (iii) can be established by a
well-known argument concerning the Bergman metric. Our proof of
Theorem~A (presented in sections 2--5 below) deals therefore mainly
with the implication (iii) $\Rightarrow$ (i). The crucial technical
ingredient in the proof of this result are certain estimates for the
Bergman kernel $K_\Psi(z,w)$. Such estimates have previously been
obtained by F. Holland and R. Rochberg in \cite{HR}. The results of
\cite{HR} are not directly applicable because we need more precise
off-diagonal estimates for the kernel than those given in that
paper. Our method of proof is similar to that of \cite{HR}, but our
approach highlights more explicitly the interplay between the
smoothness of $\Psi$ and the off-diagonal decay of the Bergman
kernel. This is where the additional smoothness condition
\eqref{basic} comes into play; many of our estimates can be
performed with sufficient precision without the assumption that
\eqref{basic} holds, but some condition of this kind seems to be
needed for our off-diagonal estimates.

The fact that the Bergman metric is the notion used to define the
Bloch space $\M$ suggests that Theorem~A should be extendable beyond
the case of radial weights. To obtain such an extension, one would
need a replacement of our Fourier-analytic approach, which relies
crucially on the representation of the Bergman kernel as a power
series.

The machinery developed to prove Theorem~A leads with little extra
effort to a characterization of compact Hankel operators in terms of
the obvious counterparts to $\vmoa$ and the little Bloch space; see
Section~\ref{compact} for details. In our study of Schatten class
Hankel operators, however, some additional techniques will be used.
We will need more precise local information about the Bergman
metric, namely that balls of fixed radius in the Bergman metric are
effectively certain ellipsoids in the Euclidean metric of $\cn$ (see
Section~7). These results appear to be of independent interest; in
particular, they lead to a characterization of Carleson measures and
in turn to a characterization of the spectral properties of Toeplitz
operators (see Section~8). Building on these results and using $L^2$
estimates for the $\overline{\partial}$ operator, we obtain in
Section~9 a characterization of Schatten class Hankel operators.

To place the present investigation in context, we close this
introduction with a few words on the literature. Boundedness and
compactness of Hankel operators with arbitrary symbols have
previously been considered only for the classical Fock space ($\Psi$
a linear function); see for example \cite{Ba1}, \cite{Ba2},
\cite{BC1}, \cite{BCZ}, \cite{St}, \cite{XZ}. The methods of these
papers, relying on the transitive self-action of the group $\cn$,
can not be extended beyond this special case. Hankel operators with
anti-holomorphic symbols defined on more general weighted Fock
spaces were studied recently in \cite{BY1} and \cite{BY2}, where it
was shown that anti-holomorphic polynomials do not automatically
induce bounded Hankel operators. For Bergman kernel estimates in
similar settings, we refer to \cite{Kr} and \cite{MO}. We finally
mention \cite{JPR} and \cite{BL}; the first of these papers focuses
on small Hankel operators and the Heisenberg group action, while the
second deals with Hankel operators for the Bergman projection on
smoothly bounded pseudoconvex domains in $\cn$.

A word on notation: Throughout this paper, the notation
$U(z)\lesssim V(z)$ (or equivalently $V(z)\gtrsim U(z)$) means that
there is a constant $C$ such that $U(z)\leq CV(z)$ holds for all $z$
in the set in question, which may be a space of functions or a set
of numbers. If both $U(z)\lesssim V(z)$ and $V(z)\lesssim U(z)$,
then we write $U(z)\simeq V(z)$.

\section{General arguments: ({\rm i}) $\Rightarrow$ ({\rm ii}) and ({\rm ii})
$\Rightarrow$ ({\rm iii}) in Theorem~A}

The following standard argument shows that (i) implies (ii) in
Theorem A. To begin with, we note that if $f$ is in ${\mathcal
A}^2(\psi)$, then $\tilde f = f.$ Moreover, by the definition of the
reproducing kernel, a computation shows that
\begin{equation}\label{general} \widetilde{|f|^2}(z) - |f(z)|^2 =
\int_{\cn} |f (\xi)- f(z)|^2 \frac{\left|
K_\Psi(\xi,z)\right|^2}{K_\Psi(z,z)} d\mu_\Psi(\xi)
  = \frac{\|H_{\bar f}  K_\Psi(\cdot,z)\|^2}{K_\Psi(z,z)}.
\end{equation} Hence, if $H_{\bar f} $ is bounded, then $\| f \|_{\bmo} <
+\infty.$

The implication (ii) $\Rightarrow$ (iii) is a consequence of the
following lemma, the proof of which is exactly as the proof of
Corollary 1 in \cite{BBCZ} (see pp. 319--321 in that paper).

\begin{lemma} \label{bergmanbloch}
Suppose that $f$ is in $\bmoa(\Psi)$. Then for every piecewise
$C^1$-smooth curve $\gamma:\ [0, 1] \to \cn$ we have
$$\left |\frac{d}{dt}(f \circ \gamma) (t) \right|
\leq 2 \sqrt{2} \beta(\gamma(t), \gamma'(t)) (\mo f)(\gamma(t)).$$
\end{lemma}

If we choose $\gamma(t) = z+ t\xi$, then we obtain
\begin{equation}\label{bmoatobloch}\frac{|\langle (\nabla f ) (z), \overline{\xi} \rangle|}{  \beta(z,
\xi)}\leq 2 \sqrt{2}  (\mo f)(z)\end{equation} for all $z$ in $\cn $
and $\xi$ in $\cn \setminus \{0\}.$

\section{Estimates for the Bergman kernel and some related functions}

This section is a somewhat elaborate preparation for the proof of
Theorem~B and also the proof of the implication (iii) $\Rightarrow$
(i) in Theorem A.

Set
$$\theta_0(r) := [r\Phi'(r)]^{-1/2}.$$
The key estimates for the Bergman kernel are the following.
\begin{lemma}\label{lemmah}
Suppose that \eqref{basic} holds. Let $z$ and $w$ be arbitrary
points in $\cn$ such that $\langle z,w\rangle \neq 0$, and write
$\langle z,w\rangle=re^{i\theta}$, where $r> 0$ and $-\pi <\theta\le
\pi$. Then we have
\[ \frac{1}{[\Psi'(r)]^{n-1}} \frac{|K_\Psi(z,w)|}{e^{\Psi(r)}}\lesssim \begin{cases} \Phi'(r), &
|\theta|\le  \theta_0(r)  \\
r^{-3/2}[\Phi'(r)]^{-1/2}|\theta|^{-3}, & |\theta|>
\theta_0(r).\end{cases}\] Moreover, there exists a positive constant
$c$ such that if $\theta<c\theta_0(r)$, then
\[ |K_\Psi(z,w)|\gtrsim
\Phi'(r)[\Psi'(r)]^{n-1}e^{\Psi(r)}.\]
\end{lemma}

We collect a few preliminary results.

\begin{lemma}\label{smooth}
Let $\eta$ be as in Theorem~A. Then, for any fixed $\alpha>\eta$, we
have
\[ \sup_{|\tau|\le t^{1/2} [\Phi'(t)]^{-\alpha}} \Phi'(t+\tau)=(1+o(1))\Phi'(t)\]
when $t\to\infty$.
\end{lemma}
\begin{proof}
The proof is similar to the proof of Lemma 6 in \cite{HR}. By
\eqref{basic}, $[\Phi'(x)]^{-1-\eta}\Phi''(x)=O(x^{-1/2})$ when
$x\to\infty$, which implies that
\[ |[\Phi'(t+\tau)]^{-\eta}-[\Phi'(t)]^{-\eta}|=|\tau| O(t^{-1/2}\tau)\]
when $t\to\infty$. The result follows from this relation.
\end{proof}

In order to estimate $|K_\Psi(z,w)|$, we need precise information
about the moments $s_d$. To this end, note that the integrand of
\[\int_{0}^\infty x^t e^{-\Psi(x)} dx \] attains its maximum at
$x=\Phi^{-1}(t)$. Set
\[ h_t(x)=-t\log x +\Psi(x)-(-t\log {\Phi^{-1}}(t)+\Psi(\Phi^{-1}(t)))\]
and
\[ I(t)=\int_0^\infty e^{-h_t(x)}dx; \]
we may then write
\[ s_d=e^{d\log \Phi^{-1}(d)-\Psi(\Phi^{-1}(d))}I(d).\]
We have the following precise estimate for $I(t)$.
\begin{lemma}\label{I-est}
For the function $I(t)$, we have
\[
I(t)=(\sqrt{2\pi}+o(1))\left[\frac{\Phi^{-1}(t)}{\Phi'(\Phi^{-1}(t))}\right]^{1/2}
\]
when $t\to\infty$.
\end{lemma}
\begin{proof}
Set $\tau(x)=\sqrt{x}[\Phi'(x)]^{-\alpha}$, where $\eta<\alpha<1/2$.
Since
\[ h_t''(x)=\frac{\Phi'(x)}{x}+\frac{t}{x^2}-\frac{\Phi(x)}{x^2}
=\frac{\Phi'(x)}{x}+\frac{1}{x^2}\big[\Phi(\Phi^{-1}(t))-\Phi(x)\big],\]
we have, by Lemma~\ref{smooth}, \[
h_t''(x)=h_t''(\Phi^{-1}(t))(1+o(1))\] when $|x-\Phi^{-1}(t)|\le
\tau(\Phi^{-1}(t))$.  On the other hand, by the convexity of $h_t$,
we then have
\[ |h_t(x)|\ge
\frac{1}{2}(h_t''(\Phi^{-1}(t))+o(1))\tau(\Phi^{-1}(t))|x-\Phi^{-1}(t)|\]
for $|x-\Phi^{-1}(t)|\ge \tau(\Phi^{-1}(t))$. Setting for simplicity
\[ c=h_t''(\Phi^{-1}(t))=\frac{\Phi'(\Phi^{-1}(t))}{\Phi^{-1}(t)}, \] we then get
\begin{equation}\label{gauss} I(t)=\int_{|x|\le \tau(\Phi^{-1}(t))}
e^{-\frac{1}{2}(c+o(1))x^2}dx+E(t),\end{equation} where
\[ |E(t)|\le 2\int_{x\ge \tau(\Phi^{-1}(t))}
e^{-\frac{1}{2}(c+o(1)){ \tau(\Phi^{-1}(t}))x}dx
%=\frac{4[\Phi^{-1}(t)]^{1/2}}{ [\Phi'(\Phi^{-1}(t))]^{1-\alpha}}
%e^{-(\frac{1}{2}+o(1))[\Phi'(t)]^{1-2\alpha}}
.\] Thus the result
follows, since the integral in \eqref{gauss} can be estimated by the
corresponding Gaussian integral from $-\infty$ to $\infty$.
\end{proof}

In what follows, we will estimate a number of integrals in a similar
fashion, using Lemma~\ref{smooth} to split the domain of
integration. The integrands will be of the type $e^{-g_t(x)}S_t(x)$
and satisfy the following:
\begin{itemize}
\item[(I)] $g_t$ attains its minimum at a point $x_0=x_0(t)\to \infty$ with $g_t''(x)=(1+o(1))
c$ for $|x-x_0|\le \tau$ and $1/\tau=o(c)$ when $t\to\infty$.
\item[(II)] For $|x-x_0|\le \tau$, $S_t(x)$ can be estimated by a constant $C$ times $|x-x_0|^m$
for some positive integer $m$.
\item[(III)] When $|x-x_0|\ge \tau$ and $|x-x_0|$ grows, the function $e^{-g_t(x)} S_t(x)$
decays so fast that
\[ \int_{0}^\infty e^{-g_t(x)}|S_t(x)| dx=(1+o(1))\int_{|x-x_0|\le \tau}
e^{-g_t(x)}|S_t(x)|
dx. \]
\end{itemize}
Taking into account the formula \begin{equation} \label{formula}
\int_0^\infty x^m e^{-\frac{1}{2}cx^2} dx =
(c/2)^{-(m+1)/2}\int_0^\infty x^me^{-x^2} dx, \end{equation} we then
arrive at the estimate
\begin{equation}\label{scheme}
\int_0^\infty e^{-h_t(x)} S_t(x) dx =O(C c^{-(m+1)/2})
\end{equation}
when $t\to \infty$.

We will at one point encounter a slightly different variant of this
scheme, obtained by replacing (II) by the following:
\begin{itemize}
\item[(II')] For $|x-x_0|\le \tau$, we have $S(x)=(1+o(1))(x-x_0)$ when
$t\to\infty$.
\end{itemize}
In this case, because of the symmetry around the point $x_0$, we get
the slightly better estimate
\begin{equation}\label{scheme2}
\int_0^\infty e^{-h_t(x)} S(x) dx =o(c^{-1})
\end{equation}
when $t\to\infty$.

To avoid tedious repetitions, we will in what follows omit most of
the details of such calculus arguments. We will briefly state that
conditions (I), (II), (III) (or respectively (I), (II'), (III)) are
satisfied and conclude that this leads to the estimate
\eqref{scheme} (or respectively \eqref{scheme2}).

In the proof of the next lemma, we will use this scheme three times.
\begin{lemma}\label{I-deriv}
We have
\begin{align*}
I'(t) & =
O\big([\Phi^{-1}(t)\Phi'(\Phi^{-1}(t))]^{-1/2}I(t) \big); \\
I''(t)&=
O\big([\Phi^{-1}(t)\Phi'(\Phi^{-1}(t))]^{-1}I(t) \big); \\
I'''(t) & =
O\big(\big[\Phi^{-1}(t)\Phi'(\Phi^{-1}(t))\big]^{-3/2}I(t)\big)
\end{align*}
when $t\to\infty$.
\end{lemma}
\begin{proof}
We begin by noting that $I'$ can be computed in the following
painless way:
\begin{equation} I'(t)=\int_0^\infty \log\frac{x}{\Phi^{-1}(t)}\
e^{-h_t(x)} dx \label{firstder};\end{equation} this holds because
$h_t'(\Phi^{-1}(t))=0$. For the same reason, we get
\begin{equation} I''(t)=\int_0^\infty \Big[-\frac{(\Phi^{(-1)})'(t)}{\Phi^{-1}(t)}
+\Big(\log\frac{x}{\Phi^{-1}(t)}\Big)^2\Big] e^{-h_t(x)} dx
\label{secondder}\end{equation} and
\begin{equation}\label{thirdder}
I'''(t)=\int_0^\infty
\Big[-\big[\frac{(\Phi^{(-1)})'(t)}{\Phi^{-1}(t)}\big]'-
3\frac{(\Phi^{(-1)})'(t)}{\Phi^{-1}(t)}\log\frac{x}{\Phi^{-1}(t)}
+\Big(\log\frac{x}{\Phi^{-1}(t)}\Big)^3\Big] e^{-h_t(x)} dx.
\end{equation}
We use that $[\Phi^{-1}]'(t)=1/\Phi'(\Phi^{-1}(t)$, and then in
\eqref{thirdder} we also use the fact that
\begin{equation}\label{thirdtrick}
 \Big[\frac{1}{\Phi'(\Phi^{(-1)}(t))\Phi^{-1}(t)}\Big]'=
-\frac{\Phi''(\Phi^{-1}(t))}{[\Phi'(\Phi^{-1}(t))]^3\Phi^{-1}(t)} -
\frac{1}{[\Phi'(\Phi^{-1}(t))\Phi^{-1}(t)]^2};\end{equation} we
apply condition \eqref{basic} to the first term on the right-hand
side. When we estimate the integrals in \eqref{firstder},
\eqref{secondder}, and \eqref{thirdder}, we use that
\[\big|\log\frac{x}{\Phi^{-1}(t)}\big|\le e
\frac{|x-\Phi^{-1}(t)|}{\Phi^{-1}(t)}\] for $x\ge e^{-1}
\Phi^{-1}(t)$ and that, say,
\[ \big|\log\frac{x}{\Phi^{-1}(t)}\big|\le \log\frac{1}{\Phi^{-1}(t)}\]
when $1\le x < e^{-1} \Phi^{-1}(t)$. In each case, the integrand
satisfies conditions (I), (II), (III) with $g_t=h_t$, so that we may
use \eqref{scheme}. The desired results for $I'$, $I''$, $I'''$ now
follow from \eqref{scheme}.
\end{proof}

We will need similar estimates for the function
\[ L_r(t)=\exp\big(t\log
r-t\log\Phi^{-1}(t)+\Psi\big(\Phi^{-1}(t)\big)\big),\] where $r$ is
a positive parameter.
\begin{lemma}\label{L-deriv}
We have
\begin{align*}
L_r'(t) & =
\big(-\log\frac{\Phi^{-1}(t)}{r}\big) L_r(t); \\
L_r''(t)&=  \left[\big(\log\frac{\Phi^{-1}(t)}{r}\big)^2-
\frac{1}{\Phi'(\Phi^{-1}(t))\Phi^{-1}(t)}\right] L_r(t); \\
L_r'''(t) & =  \left[\big(-\log\frac{\Phi^{-1}(t)}{r}\big)^3+
\frac{3\log\frac{\Phi^{-1}(t)}{r}}{\Phi'(\Phi^{-1}(t))\Phi^{-1}(t)}
+O\big(\big[\Phi'(\Phi^{-1}(t))\Phi^{-1}(t)\big]^{-3/2}\big)\right]
L_r(t)
\end{align*}
when $t\to\infty$.
\end{lemma}
\begin{proof}
The first and the second of these formulas are obtained by direct
computation. We arrive at the estimate for the third derivative by
again using \eqref{thirdtrick} and then applying condition
\eqref{basic}.
\end{proof}

\begin{proof}[Proof of Lemma~\ref{lemmah}]
We begin by recalling that
$$K_{\Psi}(z,w) =  k(\langle z,w\rangle), $$
where $$k(\zeta):=\frac{1}{(n-1)!} \sum_{d=n-1}^{\infty}\frac{d(d-1)
\cdots (d-n+2)}{s_d} \zeta^{d-n+1}.$$  We set $\langle z,w\rangle=r
e^{i\theta}$ and assume that $r>0 $ and $|\theta|\le \pi$. We may
then write
$$\frac{\langle z, w\rangle^{d}}{s_d} = \frac{L_r(d)}{I(d)}\exp(id\theta)$$
and hence
$$\aligned  \langle z, w\rangle^{n-1}K_{\Psi}(z,w) & =r^{n-1}\exp(i(n-1)\theta)k\big(r
e^{i\theta}\big)
\\
& = \frac{1}{(n-1)!}\sum_{d=n-1}^{\infty}d(d-1) \cdots (d-n+2)
\frac{L_r(d)}{I(d)}\exp(id\theta).
\endaligned$$

Let $\Omega(t)$ be a function in $C^3({\mathbb R})$ so that
$$\Omega(t)=  \frac{1}{(n-1)!}\frac{t(t-1) \cdots (t-n+2)L_r(t)}{I(t)}$$ for $t\ge n-1$ and $\Omega(t)=0$ for $t\le
n-2$. Then the Poisson summation formula gives \[r^{n-1}\exp(i(n-1)\theta)k\big(r
e^{i\theta}\big)=\sum_{j=-\infty}^\infty \widetilde{\Omega}(j),\]
where
\[ \widetilde{\Omega}(j)=\int_{-\infty}^\infty \Omega(t)e^{i (2\pi j+\theta)t}
dt.\] Integrating by parts, we obtain
\[ r^{n-1}|k\big(r
e^{i\theta}\big)|\le |\widetilde{\Omega}(0)|+\|\Omega'''\|_1
\sum_{j=1}^\infty\frac{2}{(2\pi)^3(j-1/2)^3}.\] Since
\[|\widetilde{\Omega}(0)|\le
\min\big(\|\Omega\|_1,|\theta|^{-3}\|\Omega'''\|_1\big),\] the proof
of the first part of the lemma is complete if we can prove that
\begin{equation}\label{L1-0} \| \Omega\|_1\lesssim(\Phi(r))^{n-1} \Phi'(r)e^{\Psi(r)} \end{equation}
and
\begin{equation} \label{L2-3} \|\Omega'''\|_1\lesssim (\Phi(r))^{n-1}
\frac{e^{\Psi(r)}}{r^{3/2}\sqrt{\Phi'(r)}}.
\end{equation}

We first estimate $\|\Omega\|_1$. We write $L_r(t)=\exp(-g_r(t))$
and claim that conditions (I), (II), (III) above hold. To see this,
we observe that, by the first formula of Lemma~\ref{L-deriv}, $L_r$
attains its maximum at $t=\Phi(r)$. Moreover, $g_r$ is a convex
function and
\[ g''_r(t)=\frac{1}{\Phi'(\Phi^{-1}(t))\Phi^{-1}(t)}.\]
Lemma~\ref{smooth} implies that \[ g''_r(t)=(1+o(1))g''_r(\Phi(r))\]
when $|t-\Phi(r)|\le \sqrt{r}[\Phi'(r)]^{1-2\alpha}$. The remaining
details are carried out as in the proof of Lemma~\ref{I-est}. Using
\eqref{scheme} with $m=0$ and Lemma~\ref{I-est}, we therefore get
\begin{align*}  \|\Omega\|_1 & = \left|\Phi(r)(\Phi(r)-1) \cdots
(\Phi(r)-n+2)
\right|\frac{L_r(\Phi(r))}{I(\Phi(r))}(\sqrt{2\pi}+o(1))[\Phi'(r)
r]^{1/2} \\
& =(1+o(1))(\Phi(r))^{n-1}\Phi'(r)e^{\Psi(r)},\end{align*}  which
shows that \eqref{L1-0} holds.

To arrive at \eqref{L2-3}, we need a pointwise estimate for
$\Omega'''$. To simplify the writing, we set
\[ a=\big|\log\frac{\Phi^{-1}(t)}{r}\big| \ \ \ \text{and}
\ \ \ b=\big[\Phi'(\Phi^{-1}(t))\Phi^{-1}(t)\big]^{-1/2}.\] Then
using the Leibniz rule along with Lemma~\ref{I-deriv} and
Lemma~\ref{L-deriv}, we get
\[ |\Omega'''(t)|\lesssim (a^3+a^2b+ab^2+b^3)\Omega(t).\]
By a straightforward calculus argument, we verify that each of the
terms in this expression satisfies (I), (II), and (III) above, again
with $x_0=\Phi(r)$ $\tau=\sqrt{r}[\Phi'(r)]^{1-2\alpha}$. We now use
\eqref{scheme} to achieve the desired estimate for each of the terms
$a^m b^{3-m}\Omega(t)$.

The previous proof also gives the second estimate when $\theta=0$,
because then $\widetilde{\Omega}(0)=\|\Omega\|_1$. To prove it in
general, we need to check that $k(r)\simeq |k(re^{i\theta})|$ when
$|\theta|\le c [r \Phi'(r)]^{-1/2}$. To this end, note that
\[
\widetilde{\Omega}(0)=e^{i\theta\Phi(r)}\int_{-\infty}^\infty
\Omega(t)e^{i\theta(t-\Phi(r))} dt,\] which implies that
\[ |\widetilde{\Omega}(0)|\ge \|\Omega\|_1-\int_{-\infty}^\infty
\Omega(t)|\theta||t-\Phi(r)| dt.\] The integral on the right is
computed using \eqref{scheme} with $m=1$, and so we get
\[ |\widetilde{\Omega}(0)|\ge
\|\Omega\|_1\big(1-C|\theta|[r\Phi'(r)]^{1/2}\big).\] Thus the
second estimate in Lemma~\ref{lemmah} holds for $c$ sufficiently
small.
\end{proof}

We close this section by proving some estimates for another function
that will be important later. Set
\begin{equation} \label{Qdef} Q_x(r)=\frac{1}{2}(\Psi(r^2)+\Psi(x^2))-\Psi(xr).
\end{equation}
\begin{lemma}\label{LemQ}
Let $\alpha$ be a positive number such that $\eta<\alpha<1/2$,  let
$x_1$ and $x_2$ be the two points such that $x_1<x<x_2$ and
\[ |x-x_1|=|x-x_2|=[\Phi'(x)]^{-\alpha},\] and set $c=\Psi'(0)$. When $r\to\infty$, we have
\begin{align}
\label{middle} Q_x''(r) & =(1+o(1))\Phi'(x^2), \ \ \ x_1\le r \le
x_2; \\
\label{small} Q_x(r) & \ge \frac{c}{4}
(x-r)^2+\big(\frac{1}{4}+o(1)\big)[\Phi'(x^2)]^{1-2\alpha}, \ \ \
r<x_1;\\ \label{large} Q_x(r)& \ge \frac{c}{4}
(x-r)^2+\big(\frac{1}{4}+o(1)\big)[\Phi'(r^2)]^{1-2\alpha}, \ \ \
r>x_2.
\end{align}
\end{lemma}
\begin{proof}
We begin by noting that
\[ Q_x'(r)=r\Psi'(r^2)-x\Psi'(xr) \] and \[
Q_x''(r)=\Psi'(r^2)+2r^2\Psi''(r^2)-x^2\Psi''(xr).\] We observe that
for $x_1\le r \le x_2$ Lemma~\ref{smooth} applies:
\[
Q_x''(r)=\Phi'(r^2)+r^2\Psi''(r^2)-x^2\Psi''(xr)=(1+o(1))\Phi'(x^2),
\]
and so we have established \eqref{middle}. For $r<x_1$, we use  the
following estimate:
\begin{align*}
Q_x(r)& \ge \frac{1}{2}\int_r^x \Psi'(s^2)(s-x) ds +
\frac{1}{2}\int_{x-[\Phi'(x^2)]^{-\alpha}}^x
\int_{x-[\Phi'(x^2)]^{-\alpha}}^t Q_x''(u)du dt \\
& \ge  \frac{c}{4}
(x-r)^2+\big(\frac{1}{4}+o(1)\big)[\Phi'(x^2)]^{1-2\alpha},
\end{align*} where we used again Lemma~\ref{smooth} in the last step.
Now observe that since $\Psi''(y)$ is a nondecreasing function, we
have
\[ Q_x''(r)\ge \Phi'(r^2)\] for $r\ge x$. We therefore obtain for
$x>x_2$:
\begin{align*} \label{large}Q_x(r)& \ge \frac{1}{2}\int_x^r
\Psi'(s^2)(s-x) ds + \frac{1}{2}\int_{r-[\Phi'(r^2)]^{-\alpha}}^r
\int_{r-[\Phi'(r^2)]^{-\alpha}}^t Q_x''(u)du dt \\
& \ge  \frac{c}{4}
(x-r)^2+\big(\frac{1}{4}+o(1)\big)[\Phi'(r^2)]^{1-2\alpha},
\end{align*} where Lemma~\ref{smooth} is applied once more. Hence
\eqref{large} also holds.
\end{proof}

\section{Proof of Theorem~B: Computation of the Bergman metric}

We begin by recalling that
\[ K_\Psi(z,z)=k(r^2),\]
where
\[k(r)=\sum_{n=0}^\infty c_d
r^d, \] and
\[ c_d:=\frac{(d+1)\cdots (d+n-1)}{(n-1)!\, s_{d+n-1}}.\]
A computation shows that
$$\beta^2(z, \xi):= |\xi|^2   \frac{k'(|z|^2)}{k(|z|^2)} +
|\langle z, \xi \rangle|^2 \left[\frac{k''((|z|^2))}{k((|z|^2))} -
\left(\frac{k'(|z|^2)}{k(|z|^2)}\right)^2\right]. $$ Thus Theorem~B
is a consequence of the following lemma.

\begin{lemma}\label{lemmahderiv2} Suppose that \eqref{basic} holds .
Then we have
$$\aligned  \frac{k'(r)}{k(r)} & =  \left(1+
o(1)\right)\Psi'(r), \\
\left(\frac{k'(r)}{k(r)}\right)' &
=(1+o(1))\Psi''(r)+o(1)\frac{\Psi'(r)}{r}
       \endaligned
   $$
   when $r\to \infty$.
\end{lemma}

The proof of this lemma relies on the following estimates.

\begin{lemma}\label{estsigma}
Suppose that \eqref{basic} holds and let the coefficients $c_d$ be
as defined above. Then we have
\begin{align}\label{firstd} \sum_{d=1}^{\infty} c_d (d-\Phi(r))r^d& =
o([r\Phi'(r)]^{1/2} k(r)),\\
\label{secondd} \sum_{d=1}^{\infty} c_d (d-\Phi(r))^2 r^d&=
(1+o(1))r\Phi'(r)k(r).
\end{align}
   when $r\to \infty$.
\end{lemma}
\begin{proof}
The proof is essentially the same as the proof for the diagonal
estimates in Lemma~\ref{lemmah}. The only difference is that we
replace the function $\Omega(t)$ by respectively
$(t-\Phi(r))\Omega(t)$ and $(t-\Phi(r))^2\Omega(t)$. In the first
case, we have a function that satisfies condition (II') in
Section~3. This means that we may use \eqref{scheme2} to arrive at
\eqref{firstd}. To establish \eqref{secondd}, may we apply
\eqref{formula} with $m=2$ and take into account that we have the
explicit factor $(t-\Phi(r))^2$ in front of $\Omega(t)$.
\end{proof}

\begin{proof}[Proof of Lemma~\ref{lemmahderiv2}]
We write
\[ k'(r)=\frac{\Phi(r)}{r}(k(r)+O(1))+\frac{1}{r}\sum_{d=1}^\infty
c_d(d-\Phi(r))r^d;
\]
using Lemma~\ref{estsigma}, we obtain
\[ \frac{k'(r)}{k(r)}=\left(1+
o(1)\right)\Psi'(r)+o\left(\left[\frac{\Phi'(r)}{r}\right]^{1/2}\right).\]
The desired estimate for $k'/k$ follows because, in view of
Lemma~\ref{smooth}, we have
\[\Phi(r)\ge \int_{r-r^{1/2}[\Phi'(r)]^{-\alpha}}^r \Phi'(t) dt
=(1+o(1))r^{1/2} [\Phi'(r)]^{1-\alpha}\] for some $\alpha<1/2$.

To arrive at the second estimate, we first observe that
\[ \aligned k''(r)& = \frac{\Phi(r)-1}{r}(k'(r)+O(1))+
\frac{1}{r}\sum_{d=2}^\infty c_d d(d-\Phi(r))r^{d-1} \\
& =
\frac{\Phi(r)-1}{r}(k'(r)+O(1))+\frac{\Phi(r)}{r^2}\sum_{d=2}^\infty
c_d (d-\Phi(r)) r^d + \frac{1}{r^2}\sum_{d=2}^\infty c_d
(d-\Phi(r))^2 r^{d}. \endaligned \] Combining our expressions for
$k'$ and $k''$, we find that
\[ \aligned k''(r)k(r)-(k'(r))^2=& \frac{k(r)}{r^2}\sum_{d=2}^\infty c_d
(d-\Phi(r))^2 r^{d} -\frac{1}{r^2}\left[\sum_{d=2}^\infty c_d
(d-\Phi(r)) r^d\right]^2
\\ & -\frac{k(r)k'(r)}{r}+\Psi'(r)O(k(r)+k'(r)).
\endaligned\]
Using again Lemma~\ref{estsigma} and the estimate already obtained
for $k'/k$, we get \[
\left(\frac{k'(r)}{k(r)}\right)'=(1+o(1))\frac{\Phi'(r)}{r} -
(1+o(1))\frac{\Phi(r)}{r^2}\] from which the second estimate in
Lemma~\ref{lemmahderiv2} follows.
\end{proof}

\section{Hankel operators from Bloch functions}

We finally turn to the proof that (iii) implies (i) in Theorem A. A
different proof, using $L^2$ estimates for the $\overline{\partial}$
operator will be given in Section~9 below, subject to an additional
mild smoothness condition on $\Psi$. The proof in Section~9 gives a
more informative norm estimate, which will be crucial in our study
of Schatten class Hankel operators. The proof to be given below has
the advantage that it does not require $f$ to be holomorphic.

Using the reproducing formula, we find that
$$ \aligned
H_{\bar f}g (z) = \int_{\cn}  \left(\overline{ f (z)} - \overline{
f(w)}\right) K_\Psi(z,w)g(w) d\mu_\Psi(w).
\endaligned
$$
Therefore, by the definition of $\M$, we have
$$|H_{\bar f}g (z)| \leq \| f \|_{\M}  \int_{\cn}  \varrho(z, w)K_\Psi(z,w)g(w) d\mu_\Psi(w).$$ Thus it
suffices to prove that the operator $A$ defined as
\[ Ag(z)=\int_{\cn}  \varrho(z, w)K_\Psi(z,w)g(w)
d\mu_\Psi(w)\] is bounded on $L^2(\mu_{\Psi})$.

We shall use a standard technique known as Schur's test \cite[p.
42]{Z}. Set
\[ H(z,w)=
\varrho(z,w)|K_\Psi(z,w)|e^{-\frac{1}{2}(\Psi(|z|^2)+\Psi(|w|^2))}.\]
By the Cauchy--Scwharz inequality, we obtain
\[ |(Ag)(z)|^2e^{-\Psi(|z|^2)}\lesssim \int_{\cn} H(z,\zeta) dV(\zeta)
\int_{\cn} H(z,w) |g(w)|^2 e^{-\Psi(|w|^2)}dV(w).\] This means that
the operator $A$ is bounded on $L^2(\mu_{\Psi})$ if
\begin{equation}\label{keybound}
\sup_{z} \int_{\cn} H(z,\zeta) dV(\zeta)<\infty.
\end{equation}
We therefore set as our task to establish \eqref{keybound}.

Without loss of generality, we may assume that $z=(x,0,...,0)$ with
$x>0$. We begin by estimating $\varrho(z,w)$. To this end, write
$w=(w_1,\xi)$ with $\xi$ a vector in $\C^{n-1}$ and
$w_1=re^{i\theta}$ when $n>1$. Set $e_1=(1,0,...,0)$ and consider
the three curves
\[ \aligned \gamma_1 (t) &  = xe^{it} e_1, \ \ 0 \leq t\leq \theta, \\
\gamma_2 (t) &  = (x+t(r-x)) e^{i\theta} e_1, \ \ 0 \leq t \leq  1,\\
\gamma_3 (t) &  = (r e^{i\theta} ,t \xi), \ \ 0 \leq t \leq  1,
\endaligned\]
which together constitute a piecewise smooth curve from $z$ to $w$.
(When $n=1$, $\gamma_3$ does not appear and can be neglected.) Note
that
\[ \aligned
|\langle \gamma_1(t), \gamma_1'(t)\rangle|& = |\gamma_1(t)||\gamma'_1(t)|= x^2, \\
|\langle \gamma_2(t), \gamma_2'(t)\rangle| & =|\gamma_2(t)||\gamma'_2(t)|=
(x+t(r-x))|x-r|, \\
|\langle \gamma_3(t), \gamma_3'(t)\rangle| & = t |\xi|^2.\endaligned
\] By these observations and Theorem~B, we get the following
estimate:
\[ \aligned \varrho(z, w) \ \lesssim & \  x|\theta| [\Phi'(x^2)]^{1/2} +
[\Phi'(\max(x^2,r^2))]^{1/2} |x-r| \\
& + |\xi| [\Psi'(r^2+|\xi|^2)]^{1/2} + |\xi|^2
[\Psi''(r^2+|\xi|^2)]^{1/2}. \endaligned \] When estimating the last
term on the right-hand side of this inequality, we will use that
\begin{equation} \label{psideriv} [\Psi'(y)]^2\gtrsim \Psi''(y),
\end{equation} which is a consequence of our assumptions
\eqref{admissible} and \eqref{basic}. Indeed, assuming $\Psi''>0$,
we have $y \Psi''(y)\simeq \Phi'(y)$ since $\Psi''$ is a
nondecreasing function. Thus \eqref{psideriv} is equivalent to the
following:
\[ \Phi(t)\gtrsim t^{1/2} [\Phi'(t)]^{1/2}.\]
We arrive at this estimate because
\[ \Phi(t)=\Phi(0)+\int_0^t \Phi'(\tau)d\tau\ge
\Phi(0)+(1+o(1))t^{1/2}[\Phi'(t)]^{1/2}, \] where in the second step
we used Lemma~\ref{smooth} with $\alpha=1/2$.

For $\zeta=|\zeta|e^{i\theta}$, we set
\[
 h(\zeta)=
\begin{cases} \Phi'(|\zeta|), &
|\theta|\le  \theta_0(|\zeta|)  \\
|\zeta|^{-3/2}[\Phi'(|\zeta|)]^{-1/2}|\theta|^{-3}, & |\theta|>
\theta_0(|\zeta|)\end{cases}.\] Using this notation and
Lemma~\ref{lemmah}, we then obtain
\[ H(z,w)\lesssim \varrho(x,w)
h(xre^{i\theta})[\Psi'(xr)]^{n-1}
e^{-\frac{1}{2}(\Psi(x^2)+\Psi(r^2+|\xi|^2))-\Psi(xr)}.\]  By
Fubini's theorem, we may compute the integral in \eqref{keybound} by
first integrating with respect to the vector $\xi$ over $\C^{n-1}$
and then taking an area integral with respect to the complex
variable $w_1$ over $\C$. Since $y\mapsto \Psi(r^2+y^2)$ attains its
maximum at $y=0$ and has a second derivative larger than
$2\Psi'(r^2)$, we have that $\Psi(r^2+y^2)-\Psi(r^2)\ge
\Psi'(r^2)y^2$. Using spherical coordinates along with this fact, we
find that
\[ \int_{\C^{n-1}} e^{-\Psi(r^2+|\xi|^2)} dV_{n-1}(\xi)
\lesssim e^{-\Psi(r^2)} [\Psi'(r^2)]^{-n+1}.\] Similarly, using
again spherical coordinates, we get
\[ \int_{\C^{n-1}}\Theta(r,|\xi|) e^{-\Psi(r^2+|\xi|^2)} dV_{n-1}(\xi)
= C \int_0^\infty \Theta(r,y)y^{2n-2}e^{-\Psi(r^2+y^2)}dy,\] where
$C$ is the surface area of the unit sphere in $\C^{n-1}$ and
$\Theta$ is any suitable function of two variables. From the
estimate for $\varrho(z,w)$ and \eqref{psideriv} we see that we are
interested in the following two choices: (1)
$\Theta(r,y)=y[\Psi'(r^2+y^2)]^{1/2}$ and (2) $\Theta(r,y)=y^2
\Psi(r^2+y^2)$. In case (1), we use the Cauchy--Schwarz inequality,
so that we get
\[ \int_{\C^{n-1}}|\xi|[\Psi'(r^2+|\xi|^2)]^{1/2} e^{-\Psi(r^2+|\xi|^2)} dV_{n-1}(\xi)
\lesssim e^{-\Psi(r^2)}\left[\int_0^{\infty} y^{4n-3}
e^{-(\Psi(r^2+y^2)-\Psi(r^2))} dy\right]^{1/2}.\] Estimating
$\Psi(r^2+y^2)-\Psi(r^2)$ as above, we therefore get
\[ \int_{\C^{n-1}}|\xi|[\Psi'(r^2+|\xi|^2)]^{1/2} e^{-\Psi(r^2+|\xi|^2)} dV_{n-1}(\xi)
\lesssim e^{-\Psi(r^2)}[\Psi'(r^2)]^{-n+1}.\] In case (2), we
integrate by parts and get
\[ \int_{\C^{n-1}}|\xi|^2 \Psi'(r^2+|\xi|^2) e^{-\Psi(r^2+|\xi|^2)} dV_{n-1}(\xi)
\lesssim \int_0^{\infty} y^{2n-1} e^{-\Psi(r^2+y^2)} dy.\] We
proceed as above and obtain
\[ \int_{\C^{n-1}}|\xi|^2 \Psi'(r^2+|\xi|^2) e^{-\Psi(r^2+|\xi|^2)} dV_{n-1}(\xi)
\lesssim e^{-\Psi(r^2)}[\Psi'(r^2)]^{-n+1}.\]

With $\sigma$ denoting Lebesgue measure on $\C$, we therefore get
\[ \int_{\cn} H(z,w)dV(w)\lesssim \int_{\C} G(x,r,\theta)
\left[\frac{\Psi'(rx)}{\Psi'(r^2)}\right]^{n-1}h(xre^{i\theta})e^{-Q_x(r)}d\sigma(re^{i\theta}),\]
where
\[ G(x,r,\theta)=x|\theta|[\Phi'(x^2)]^{1/2} + [\Phi'(\max(x^2,r^2))]^{1/2}|x-r|+
1\] and $Q_x$ is as defined by \eqref{Qdef}.

We now resort to polar coordinates; simple calculations show that
\[  \int_{-\pi}^{\pi} h(xre^{i\theta}) d\theta \lesssim
\left[\frac{\Phi'(xr)}{xr}\right]^{1/2} \ \ \ \text{and} \ \ \
\int_{-\pi}^{\pi}|\theta| h(xre^{i\theta}) d\theta \lesssim
\frac{1}{xr}\] so that \[ \int_{\cn} H(z,w)dV(w)\lesssim
\int_{0}^\infty (S_{x}(r)+T_x(r)) e^{-Q_x(r)}rdr,\] where
\[ S_x(r) =
\left(\frac{[\Phi'(x^2)]^{1/2}}{r}+
\left[\frac{\Phi'(xr)}{xr}\right]^{1/2}\right)\left[\frac{\Psi'(rx)}{\Psi'(r^2)}\right]^{n-1}
\] and \[ T_x(r) =\varphi(\max(x^2,r^2))|x-r|
\left[\frac{\Phi'(xr)}{xr}\right]^{1/2}\left[\frac{\Psi'(rx)}{\Psi'(r^2)}\right]^{n-1}.\]
By Lemma~\ref{LemQ} and a straightforward argument, we find that
both $S_xe^{-Q_x}$ and $T_xe^{-Q_x}$ satisfy conditions (I), (II),
(III) of Section 3 (with $x=t$, $Q_x=g_t$, $x_0=x$, and
$\tau=[\Phi'(x)]^{-\alpha}$). Hence \eqref{scheme} applies with
$m=0$ and $m=1$ for the respective integrands, so that we get
\[ \sup_{x>0} \int_{0}^\infty S_x(r)
 e^{-Q_x(r)}rdr<\infty \ \ \ \text{and} \ \ \ \sup_{x>0} \int_{0}^\infty T_x(r)
 e^{-Q_x(r)}rdr<\infty.
\] We may therefore conclude that \eqref{keybound} holds.
\section{Compactness of Hankel
operators}\label{compact}

We now turn to a study of the relation between the spectral
properties of Hankel operators and the asymptotic behavior of their
symbols. We begin with the case of compact Hankel operators.

An entire function is said to be of vanishing mean oscillation with
respect to $\Psi$ if $(\mo f)(z) = o(1) $ as $|z|\to + \infty.$
Entire functions of vanishing mean oscillation form a closed
subspace of $\bmoa(\Psi)$ which we will denote by $\vmoa(\Psi)$. In
accordance with our preceding discussion, we define the little Bloch
space $\littlebloch$ as the collection of functions $f$ in $\M$ for
which
$$\sup_{\xi \in \cn \setminus \{0\}}
\frac{|\langle \nabla f(z), \overline{\xi} \rangle|}{\beta(z, \xi)}
= o(1) \ \ \ \text{when} \ \  |z|\to + \infty.$$  The main result of
this section reads as follows.
\newtheorem*{thC}{\bf Theorem C}
\begin{thC}  Let $\Psi$ be a logarithmic growth function,  and
suppose that there exists a real number $\eta<1/2$ such that
\eqref{basic} holds. If  $f$ is an entire function on $\cn$, then
the following statements are equivalent:
\begin{itemize}
\item[(i)]The function $f$ belongs to $\cT(\Psi)$ and the
Hankel operator $H_{\bar f} $ on $\A$ is compact;
% from ${\mathcal A}^2(\Psi)$ into $L^2( \mu_\Psi)$;
\item[(ii)] The function $f$ belongs to $\vmoa(\Psi)$;
\item[(iii)] The function $f$ belongs to $\littlebloch$.
\end{itemize}
\end{thC}

Our proof of Theorem~C requires the following two lemmas.
\begin{lemma}\label{w-cvge}
%If $f$ belongs to $\cT(\Psi)$ and the Hankel operator $H_{\bar f} $
%is compact from ${\mathcal A}^2(\Psi)$ into $L^2( \mu_\Psi)$, then
The normalized Bergman kernels $K_\Psi(\cdot, z)/\sqrt{K_\Psi(z, z)}
$ converge weakly to $0$ in $\A$ when $|z| \to + \infty.$
\end{lemma}
\begin{proof}
Since the holomorphic polynomials are dense in $\A$, it suffices to
show that for any non-negative integer $m$, we have
$$ \frac{ |z|^m}{\sqrt{K_\Psi(z, z)}}  \to 0$$
as $|z| \to +\infty.$  But this holds trivially because
$K_\Psi(z,z)$ is an infinite power series in $|z|^2$ with positive
coefficients.
\end{proof}

 \begin{lemma}\label{approx-lbloch}
Let $f : \cn \to \C$ be a function for which there exist positive
numbers $R$ and $\varepsilon$ such that
 $$|f(z) - f(w)| \leq \varepsilon \varrho(z,w)$$
 whenever $|z|\ge R.$ Then there exists a function
 $f_{0} :  \cn \to \C$ such that $f(z) = {f_0}(z)$ for  $|z| \geq R$ and
 $$|f_0(z) - f_0(w)| \leq \varepsilon \varrho(z,w)$$
for all points $z$ and $w$ in  $\cn.$
\end{lemma}
\begin{proof} We argue as in the proof of Lemma 5.1 in  \cite{Ba1}.
We assume without loss of generality that $f$ is real-valued and set
$$f_0(z) := \inf_{w\in \cn} \{ f(w) +    \varepsilon \varrho(z,w)\}.$$
Then a straightforward argument using the triangle inequality for
the Bergman metric shows that $f_0$ has the desired properties.
 \end{proof}

\begin{proof}[Proof of Theorem~C]
We first prove the implication (i) $\Rightarrow$ (ii). Assuming that
$H_{\bar f}$ is compact, we obtain, using Lemma~\ref{w-cvge}, that
$$ [(\mo f)(z)]^2
  = \frac{\|H_{\bar f}  K_\Psi(\cdot,z)\|^2}{K_\Psi(z,z)} \to 0
$$
when $|z| \to +\infty$. This gives the desired conclusion.

We next note that the implication (ii) $\Rightarrow$ (iii) is
immediate from \eqref{bmoatobloch}.

Finally, to prove that (iii) implies (i), in view of Theorem~A, we
only need to prove that the bounded Hankel operator
$H_{\overline{f}}$ is compact whenever (iii) is satisfied. To see
that this holds, we choose an arbitrary positive $\varepsilon$.
Assuming (iii), we may find a positive $R_0$ such that
$$|\langle (\nabla f ) (z), \overline{\xi} \rangle|\leq   \frac{\varepsilon}{2}  \beta(z, \xi)$$
whenever $|z| \geq R_0 $ and $\xi$ is in  $\cn \setminus \{0\}.$
Then for some $R>R_0$ we have
$$|f(z) - f(w)| \leq \varepsilon \varrho(z,w)$$
as long as $|z|\geq R$. Indeed, this follows because
$\beta(z,\xi)/|\xi|\to\infty$ when $|z|\to\infty$ so that, whenever
$|z|$ is sufficiently large, $\varrho(z,w)$ is ``essentially''
determined by the contribution to the integral in
\eqref{bergmanmetric} from the points that lie outside the ball of
radius $R_0$ centered at $0$. Now let $f_0$ be the function obtained
from Lemma~\ref{approx-lbloch}. We write
 $$H_{\bar f} =  H_{\bar f - {\bar f}_0}  + H_{{\bar f}_0}  $$
and observe  that $\bar f - {\bar f}_0$ is a compactly supported
continuous function on $\cn$. Hence $H_{\bar f - {\bar f}_0}$ is
compact. On the other hand, if $ g$ a holomorphic polynomial, then
$$ \aligned
\left| H_{\bar {f_0}}g (z)\right| & \lesssim \int_{\cn}  \left|\bar f_0 (w) - \bar f_0(z)\right|
 \left|K_\Psi(z,w)g(w)\right| d\mu_\Psi(w) \\
&  \leq  \varepsilon \int_{\cn}
 \beta(z, \xi) \left|K_\Psi(z,w)g(w)\right| d\mu_\Psi(w)
\endaligned
$$
so that, by the proof of Theorem~A, we see that $\|H_{\bar {f_0}}\|
\lesssim \varepsilon.$ The implication (iii) $\Rightarrow$ (i)
follows because $\varepsilon$ can be chosen arbitrarily small.
 \end{proof}

\section{The geometry of Bergman balls of fixed radius}

In what follows, we will need the analogue of Lemma~\ref{smooth} for
the function $\Psi$ when $n>1$. We will therefore assume that
\begin{equation}\label{basic2}
\Psi''(t)=O(t^{-\frac{1}{2}}[\Psi'(t)]^{1+\eta}) \ \ \ \text{when} \
t\to\infty \end{equation} for some $\eta<1/2$ whenever $n>1$. This
is again a mild smoothness condition on $\Psi$.

\begin{lemma}\label{smooth2}
Assume that \eqref{basic2} holds for some $\eta<1/2$. Then, for any
fixed $\alpha>\eta$, we have
\[ \sup_{|\tau|\le t^{1/2} [\Psi'(t)]^{-\alpha}} \Psi'(t+\tau)=(1+o(1))\Psi'(t)\]
when $t\to\infty$.
\end{lemma}

We are interested in describing geometrically the Bergman ball
\[ B(z,a)=\{w:\ \varrho(z,w)<a\}. \]
Let $P_z$ denote the orthogonal projection in $\cn$ onto the complex
line $\{\zeta z: \zeta \in \C\}$, where $z$ is an arbitrary point in
$\cn\setminus \{0\}$. It will be convenient to let $P_0$ denote the
identity map. We use the notation
\[ D(z,a)=\left \{w:\ |z-P_zw|\le a[\Phi'(|z|^2)]^{-1/2}, \ |w-P_z
w| \le a[\Psi'(|z|^2)]^{-1/2}\right\}. \] Then we have the following
result.
\begin{lemma}\label{bergmannball}
Suppose that there exists a real number $\eta<1/2$ such that
\eqref{basic} holds and that \eqref{basic2} holds if $n>1$. Then,
for every positive number $a$, there exist two positive numbers $m$
and $M$ such that
\[ D(z,m)\subset B(z,a) \subset D(z,M) \]
for every $z$ in $\cn$.
\end{lemma}

\begin{proof}
It suffices to prove that \begin{equation}\label{suffice}
\varrho(z,w)\simeq |z-P_z w| [\Phi'(|z|^2)]^{1/2}+|w-P_z w|
[\Psi'(|z|^2)]^{1/2}\end{equation} for $w$ in $D(z,M)$ for any fixed
positive number $M$. (The latter term vanishes and can be
disregarded when $n=1$.) To begin with, we note that Theorem~B gives
that
\begin{equation}\label{firstvarrho} \varrho(z,w)\simeq \inf_{\gamma} \int_{0}^1
\left( |\gamma'(t)|[\Psi'(|\gamma(t)|^2)]^{1/2}+|\langle
\gamma(t),\gamma'(t)\rangle| [\Psi''(|\gamma(t)|^2)]^{1/2}\right)dt,
\end{equation} where the infimum is taken over all piecewise smooth curves
$\gamma: [0,1]\to \cn$ such that $\gamma(0)=z$ and $\gamma(1)=w$. If
we choose $\gamma$ to be the line segment from $z$ to $P_z w$
followed by the line segment from $P_z w$ to $w$ and use that
$\Psi''(x)=o([\Psi'(x)]^{1/2})$ on the latter part of $\gamma$, we
get from \eqref{firstvarrho} that
\[ \varrho(z,w)\lesssim  |z-P_z w|[\Phi'(|z|^2)]^{1/2} \\
+|P_z w-w| [\Psi'(|z|^2)]^{1/2}+|P_z w -w|^2 o(\Psi'(|z|^2)). \]
This gives the desired bound from above because, by assumption,
$|P_z w-w|\le M [\Psi'(|z|^2)]^{-1/2}$.

To prove the bound from below, we argue in the following way. Let
$\ell(\gamma)$ denote the Euclidean length of $\gamma$. Set
\[\varrho^*_{\gamma}(z,w)=\int_{0}^1
\left( |\gamma'(t)|[\Psi'(|\gamma(t)|^2)]^{1/2}+|\langle
\gamma(t),\gamma'(t)\rangle|
[\Psi''(|\gamma(t)|^2)]^{1/2}\right)dt\] and
$\varrho^*(z,w)=\inf_{\gamma}\varrho^*_{\gamma}(z,w)$. We observe
that \eqref{firstvarrho} implies that
\begin{equation}\label{frabove} \varrho(z,w) \gtrsim \inf_{t}
[\Psi'(|\gamma(t)|^2)]^{1/2} \ell(\gamma)\end{equation} whenever,
say, $\varrho^*_{\gamma}(z,w)\le 2 \varrho^*(z,w)$. Since we know by
the first part of the proof that $\varrho(z,w)\lesssim 1$, this
implies that
\[ \ell(\gamma) \lesssim \inf_{t} [\Psi'(|\gamma(t)|^2)]^{-1/2}.\]
By Lemma~\ref{smooth2}, we therefore have
\[ \ell(\gamma) \lesssim [\Psi'(|z|^2)]^{-1/2},\]
which, in view of \eqref{frabove}, in turn gives
\begin{equation}\label{boundforrho} \ell(\gamma) \lesssim [\Psi'(|z|^2)]^{-1/2}
\varrho(z,w).\end{equation}

Now let $\gamma$ be any curve such that $\varrho^*_{\gamma}(z,w)\le
2 \varrho^*(z,w)$. We then get from \eqref{firstvarrho} that
\begin{equation}\label{frbelow} \varrho(z,w) \gtrsim
|z-w|[\Psi'(|z|^2)]^{1/2} +\int_0^1 |\langle
\gamma(t),\gamma'(t)\rangle|[\Psi''(|\gamma(t)|^2)]^{1/2} dt.
\end{equation} Set $\gamma_0(t)=P_z(\gamma(t))$ and $\gamma_1(t)=\gamma(t)-\gamma_0(t)$.
Note that $\gamma_1(0)=0$ and that $\ell(\gamma_1)\le \ell(\gamma)$.
By orthogonality and the triangle inequality, we get \[ \aligned
\int_0^1 |\langle
\gamma(t),\gamma'(t)\rangle|[\Psi''(|\gamma(t)|^2)]^{1/2} dt& \ge \
\int_0^1 | \gamma_0(t)| |\gamma_0'(t)|
[\Psi''(|\gamma_0(t)|^2)]^{1/2} dt
\\ & \ \ \ \ \ \ \ \ \ \ - \int_{0}^1|\langle
\gamma_1(t),\gamma_1'(t)\rangle|[\Psi''(|\gamma(t)|^2)]^{1/2}
dt.\endaligned\] Let $t_1$ be the smallest $t$ such that
$|z-\gamma_0(t)|=|z-P_z w|$. Using that $\Psi''(x)=o([\Psi'(x)]^2)$
and \eqref{boundforrho}, we then get
\[ \aligned \int_0^1 |\langle \gamma(t),\gamma'(t)\rangle|[\Psi''(|\gamma(t)|^2)]^{1/2} dt
 & \ge   (1+o(1))\int_0^{t_1} |z| |\gamma_0'(t)| [\Psi''(|z|^2)]^{1/2}
dt - [\ell(\gamma)]^2 o(\Psi'(|z|^2)) \\
& \gtrsim |z-P_zw| |z| [\Psi''(|z|^2)]^{1/2}- o(1)
\varrho(z,w)\endaligned\] when $|z|\to \infty$. Plugging this
estimate into \eqref{frbelow}, we obtain the desired bound from
below.
\end{proof}
It follows from the previous lemma that the Euclidean volume of
$B(z, r)$ can be estimated as
\begin{equation}\label{ball-vol} |B(z, r)| \simeq
[\Phi'(|z|^2)]^{-1/2}[\Psi'(|z|^2)]^{(n-1)/2}
\end{equation}
when $r$ is a fixed positive number. We will now use this fact to
establish two covering lemmas.
\begin{lemma}\label{covering1}
Suppose that there exists a real number $\eta<1/2$ such that
\eqref{basic} holds and that \eqref{basic2} holds if $n>1$. Let $R$
be a positive number and $m$ a positive integer. Then there exists a
positive integer $N$ such that every Bergman ball $B(a, r)$ with
$r\le R $ can be covered by $N$ Bergman balls $B(a_k, \frac{r}{m})$.
\end{lemma}
\begin{proof} Fix a ball $B(a, r)$. Choose $a_0:= a$ and let $a_1$ be a point in $\cn$ such
that $\varrho(a,a_1)=r/m$. Now iterate so that in the $k$-th step
$a_k$ is chosen as a point in the complement of
$\bigcup_{j=1}^{k-1}B(a_j,r/m)$ minimizing the distance from $a$,
and let $J$ be the smallest $k$ such that $\varrho(a,a_k)\ge r$.
Then the balls $B(a_0,r/m),...,B(a_{J-1},r/m)$ constitute a covering
of $B(a,r)$. By the triangle inequality, we see that the sets
$B(a_j, r/(2m))$ are mutually disjoint, and they are all contained
in $B(a, r + r/(2m))$ when $j<J$. Hence
$$\sum_{j=0}^{J-1}
\left|B(a_j, r/(2m))\right| \leq \left|B(a, r+ r/(2m))\right|.$$ On
the other hand,  by  (\ref{ball-vol}), it follows  that there is a
positive number $C$ depending on $R$ and $m$ but not on $a$ such
that
$$ \frac{1}{C}  \left|B(a, r+ r/(2m))\right| \leq
\left|B(a_j, r/(2m))\right| $$ for every $j$.  We observe that it
suffices to take $N$ to be the smallest  positive integer larger
than or equal to $C$.
  \end{proof}
Inspired by the construction in the previous lemma, we introduce the
following notion. We say that a sequence of distinct points $(a_k)$
in $\cn$ is a $\Psi$-lattice if the there exists a positive number
$r$ such that the balls $B(a_k,r)$ constitute a covering of $\cn$
and the balls $B(a_k,r/2)$ are mutually disjoint. Replacing $a$ by,
say, $0$ and $r/m$ by $r$ in the previous proof, we have a
straightforward way of constructing a $\Psi$-lattice. Note that
since the balls $B(a_k,r/2)$ are mutually disjoint, we must have
$\varrho(a_k,a_j)\ge r$ when $k\neq j$. The number $r$, which may
fail to be unique, is called a covering radius for the
$\Psi$-lattice $(a_k)$. The supremum of all the covering radii is
again a covering radius; it will be called the maximal covering
radius for $(a_k)$.
\begin{lemma}\label{covering2}
Suppose that there exists a real number $\eta<1/2$ such that
\eqref{basic} holds and that \eqref{basic2} holds if $n>1$, and let
$R$ be a positive number. Then there exists a positive integer $N$
such that if $(a_k)$ is a $(\Psi)$-lattice with maximal covering
radius $r\le R/2$, then every point $z$ in $\cn$ belongs to at most
$N$ of the sets $B(a_k, 2r)$.
\end{lemma}
\begin{proof} Let $N$ be the
integer obtained from Lemma~\ref{covering1} for the given $R$ when
$m=4$ and assume that $z\in \bigcap_{j=1}^{N+1} B(a_{k_{j}}, 2r)$.
Then $a_{k_{j}}$ is in $B(z, 2r)$ for every $j=1, \cdots, N+1$. If
the sets $B(z_{1}, r/2), ...,$ $B(z_{N},r/2)$ constitute a covering
of $B(z, 2r)$, the existence of which is guaranteed by
Lemma~\ref{covering1}, then at least one of the sets $B(z_{k}, r/2)$
must contain two of the points $a_{k_{j}}, j=1, \cdots, N+1$. On the
other hand, by the triangle inequality, we have reached a
contradiction because the minimal distance between any two points in
the sequence $(a_k)$ can not be smaller than $r$.
\end{proof}

 \section{Carleson measures and  Toeplitz operators }

For a nonnegative Borel measure $\nu$ on $\cn$, we set
\[ d\nu_{\Psi}(z)=e^{-\Psi(|z|^2)}d\nu(z).\]
Such a measure $\nu$ is called a Carleson measure for ${\mathcal
A}^2(\Psi)$ if there is a positive constant $C$ such that
$$\int_{\cn} |f(z)|^2 d\nu_{\Psi}(z) \leq C \int_{\cn} |f(z)|^2 d\mu_\Psi(z)$$
for every function $f$ in ${\mathcal A}^2(\Psi)$. Thus $\nu$ is a
Carleson measure for $\A$ if and only if the embedding $E_\nu$ of
$\A$ into the space $L^2(\nu_{\Psi})$ is bounded.

\newtheorem*{thD}{\bf Theorem D}
\begin{thD}  Let $\Psi$ be a logarithmic growth function, and
suppose that there exists a real number $\eta<1/2$ such that
\eqref{basic} holds and that \eqref{basic2} holds if $n>1$. If $
\nu$ is a nonnegative Borel measure on $\cn$, then the following
statements are equivalent:
\begin{itemize}
\item[(i)] $\nu$ is a Carleson measure for $\A$;
\item[(ii)] There is a constant $C>0$ such that
\[ \int_{\cn}  \frac{|K_{\Psi}(w,z)|^{2}}{K(z,z)} d\nu_{\Psi}(w) \leq C \]
for every $z$ in $\cn$; \item[(iii)] For every positive number $r$,
there is a positive number $C$ such that \[\nu( B(z, r)) \leq C
|B(z, r)|\] for every $z$ in $\cn$; \item[(iv)] There exist a
$\Psi$-lattice $(a_k)$ and a positive number $C$ such that
$$\nu( B(a_k, r)) \leq C |B(a_k, r)|$$
for every point $k$, where $r$ is the maximal covering radius for
$(a_k)$.
\end{itemize}
\end{thD}

We prepare for the proof of Theorem~D by establishing the following
two lemmas.
\begin{lemma}\label{localball}
Suppose that there exists a real number $\eta<1/2$ such that
\eqref{basic} holds and that \eqref{basic2} holds if $n>1$. Then
there exists a positive number $r_0$ such that
\[ |K_{\Psi}(z,w)|^2\simeq K(z,z)K(w,w) \]
holds for $z$ and $w$ whenever $\varrho(z,w)\le r_0$.
\end{lemma}
\begin{proof}
The lemma follows from Lemma~\ref{lemmah} along with
Lemma~\ref{bergmannball}.
%applied to both sides of this estimate, we
%obtain
%\[ |f(z)|^p| e^{-\Psi(|z|^2)} \lesssim \frac{1}{ |B(z, r)|}
%\int_{B(z, m)} |f(w)|^p e^{\Psi(|z|^2)-2\Psi(|z||P_z w|)} dV(w). \]
%Using \eqref{middle} of Lemma~\ref{LemQ}, we get \[ |f(z)|^p|
%e^{-\Psi(|z|^2)} \lesssim \frac{1}{ |B(z, r)|} \int_{B(z, m)}
%|f(w)|^p e^{-\Psi(|P_z w|^2)} dV(w).\] Observing that
%\[ |\Psi(|w|^2)-\Psi(|P_z w|^2)|\lesssim 1 \]
%uniformly for $w$ in $B(z,r)$, we arrive at the desired estimate.
\end{proof}

\begin{lemma}\label{meanv} Suppose that there exists a real number $\eta<1/2$ such that
\eqref{basic} holds and that \eqref{basic2} holds if $n>1$, and let
$r_0$ be the constant from Lemma~\ref{localball}. Then there is a
constant $C$ such that
$$ |f(z)|^2 e^{-\Psi(|z|^2)} \leq \frac{C}{ |B(z, r)|} \int_{B(z, r)} |f(w)|^2
d\mu_{\Psi}(w)$$ for every entire function $f$ on $\cn$ and every
$z$ in $\cn$.
\end{lemma}
\begin{proof}
By Lemma~\ref{localball}, the holomorphic function $w\mapsto K(z,w)$
does not vanish at any point in $B(z,r)$. Thus the function
$w\mapsto |f(w)|^2 |K_{\Psi}(z,w)|^{-2}$ is subharmonic in $B(z,r)$.
Choosing $m$ as in Lemma~\ref{bergmannball}, we therefore get
$$\aligned  |f(z)|^2 |K(z,z)|^{-2} & \lesssim
\frac{1}{ |D(z, m)|} \int_{D(z, m)} |f(w)|^2 |K_{\Psi}(z,w)|^{-2}  dV(w) \\
&  \lesssim    \frac{1}{ |B(z, r)|} \int_{B(z, r)} |f(w)|^2
|K_{\Psi}(z,w)|^{-2} dV(w).
\endaligned$$
Applying Lemma~\ref{localball} to the integrand to the left and then
Lemma~\ref{lemmah} to each side, we arrive at the desired estimate.
\end{proof}
Note that, by \eqref{ball-vol}, the lemma is valid for all positive
$r$, with the additional proviso that $C$ depend on $r$.
\begin{proof}[Proof of Theorem~D] We begin by noting that the implication (i) $\Rightarrow$ (ii)
is trivial because it is just the statement that the Carleson
measure condition holds for the functions $K(\cdot,z)$. To prove
that (ii) implies (iii), we assume that (ii) holds and consider a
ball $B(z, r)$ where $r$ is a fixed positive number. Then, by
Lemma~\ref{localball} and \eqref{ball-vol}, we have
$$ \frac{1}{|B(z, r)|} \lesssim \frac{|K_{\Psi}(z,w)|^2}{K_{\Psi}(z,z)} e^{-\Psi(|w|^2)}$$
when $\varrho(z,w)\le r_0$, and therefore we obtain
$$
\frac{\nu(B(z, r))}{|B(z, r)|} \lesssim
\int_{\cn}\frac{|K_{\Psi}(z,w)|^2}{K_{\Psi}(z,z)} e^{-\Psi(|w|^2)}
d\nu(w) \leq C.
$$
The implication  (iii) $\Rightarrow$ (iv) is trivial (modulo the
existence of $\Psi$-lattices), and we are therefore done if we can
prove that (iv) implies (i). To this end, assume that (iv) holds,
and let $(a_k)$ be a $\Psi$-lattice with maximal covering radius
$r$. By Lemma~\ref{meanv}, we see that
$$\sup_{z\in B(a_k, r)}|f(z)|^2 e^{-\Psi(|z|^2)} \lesssim \frac{1}{ |B(a_k, 2r)|} \int_{B(a_k, 2r)} |f(w)|^2 d\mu_\Psi(z)$$
for every $k$. We therefore get
\[ \int_{\cn} |f(z)|^2 d\nu_{\Psi}(z)  \lesssim   \sum_{k}\int_{ B(a_k, 2r)} |f(w)|^2 d\mu_\Psi(w) \\
 \lesssim  \int_{ \cn} |f(w)|^2 d\mu_\Psi(w),
\]
where the latter inequality holds by  Lemma~\ref{covering2}.
\end{proof}

For $\nu$ a nonnegative Borel measure on $\cn$, we define the
Toeplitz operator $T_\nu$ on $\A$ in the following way:
$$(T_\nu f)(z) :=   \int_{\cn}   f(w) K_{\Psi}(z,w) e^{-\Psi(|w|^2)} d\nu(w).$$
A computation shows that $E^{\ast}_{\nu} E_\nu= T_\nu$. Thus
Theorem~D characterizes bounded Toeplitz operators. Compact Toeplitz
operators can likewise be characterized via so-called vanishing
Carleson measures; an obvious and straightforward modification of
Theorem~D gives a description of such measures. Toeplitz operators
belonging to the Schatten classes $\cS_p$ are characterized by the
following theorem.

\newtheorem*{thE}{\bf Theorem E}
\begin{thE}  Let $\Psi$ be a logarithmic growth function, and
suppose that there exists a real number $\eta<1/2$ such that
\eqref{basic} holds and that \eqref{basic2} holds if $n>1$. If $
\nu$ is a nonnegative Borel measure on $\cn$ and $p\ge 1$, then the
following statements are equivalent:
\begin{itemize}
\item[(i)] The Toeplitz operator $T_{\nu}$ on $\A$ belongs to the the Schatten class $\cS_p$;
\item[(ii)] There exists a $\Psi$-lattice $(a_k)$ such that
$$ \sum_{k =1}^{\infty} \left(\frac{\nu( B(a_k, r))}{ |B(a_k, r)|} \right)^p < +\infty,$$
 where $r$ is the maximal covering radius for $(a_k)$.
\end{itemize}
\end{thE}

%Arguing as the previous proof we have
%\begin{thm}\label{vcarl}  If $ \nu$ is Borel measure on $\cn$, then the following are equivalent
%
%(i) $\nu$ is a vanishing Carleson measure
%
%(ii)  We have that
%$$ \lim_{|z| \to + \infty} \int_{\cn}  \frac{|K_{\Psi}(z,w)|^{2}}{K_{\Psi}(z,z)} e^{-\Psi(|w|^2)} d\nu(w) = 0. $$
%
%
%(iii)  There is  $0<r< c/8$   such that
%$$ \lim_{|z| \to + \infty} \frac{\nu( B(z, r))}{ |B(z, r)|} = 0.$$
%
%
%(iv) There is  $0<r< c/8$  such that
%$$ \lim_{k \to + \infty} \frac{\nu( B(a_k, r))}{ |B(a_k, r)|} = 0,$$
% where $(a_k) $ is as in Lemma~\ref{covering2}.
% \end{thm}

For the proof of this theorem, we require the following two lemmas.

\begin{lemma}\label{boundedJ}   Suppose that $(e_j)$ is an orthonormal basis for $\A$ and
that $(a_j)$ is a $\Psi$-lattice. Then the operator $J$ on $\A$
defined by
$$Je_j  (z) :=  \frac{K_{\Psi}(z,a_j)}{\sqrt{K_{\Psi}(a_j,a_j)}}$$
is bounded.
\end{lemma}
\begin{proof}
For two arbitrary functions $f= \sum_j c_j e_j$ and $g$ in $\A$, the
reproducing formula and the Cauchy--Schwarz inequality give \[
\left|\langle Jf , g\rangle \right|^2 = \left| \sum_{j} c_j
\frac{\overline {g(a_j)}}{\sqrt{K_{\Psi}(a_j,a_j)}} \right|^2 \le
\left( \sum_{j} |c_j|^2 \right) \left(\sum_{k}\frac{|
g(a_k)|^2}{K_{\Psi}(a_k,a_k)} \right). \] If we set
$$\nu :=  \sum_{k} \frac{e^{\Psi(|a_j|^2)}
}{K_{\Psi}(a_j,a_j)} \delta_{a_{j}},$$ then we may write this
estimate as
\[ \left|\langle Jf , g\rangle \right|^2\le \|f\|_{\A}^2 \int_{\cn}
|g(z)|^2 d\nu_{\Psi}(z).\] By Theorem~D, we see that $\nu$ is a
Carleson measure, which implies that $J$ is a bounded operator on
$\A$.
\end{proof}

\begin{lemma}\label{traceT} Suppose that $T$ is a positive operator on $\A$.
Then the trace of $T$ can be computed as \[\Tr(T)   =   \int_{\cn}
\widetilde{T}(z) K_\Psi(z , z)d\mu_\Psi(z).\]
\end{lemma}
\begin{proof}
We write $K_\Psi(z,w) = \sum_{k=0}^{\infty} e_k(z)
\overline{e_k(w)},$ where $(e_k)$ is an orthonormal basis for $\A$.
The lemma is then proved by means of the following computation:
\[ \Tr(T)   =   \sum_{k=0}^{\infty} \langle T f_k,   f_k\rangle_{\A}
= \int_{\cn} \langle TK_\Psi(\cdot , z),K_\Psi(\cdot , z)
\rangle_{\A} d\mu_\Psi(z).\]
\end{proof}

%\begin{thm}\label{schattentop}  Suppose that $p\geq 1.$ If $ \nu$ is positive Borel measure on $\cn$, then the following are equivalent
%
%(i) The Toeplitz oerator $T_\nu$ is in the Schatten class $\cS_p$.
%
%(ii) There is  $0<r< c/8$  such that
%$$ \sum_{k =1}^{ + \infty} \left(\frac{\nu( B(a_k, r))}{ |B(a_k, r)|} \right)^p < +\infty.$$
% where $(a_k) $ is as in Lemma~\ref{covering2}.
%
% (iii) The function
% $$\widetilde{\nu}(z):= \int_{\cn} \frac{|K_\Psi(w , z)|^2}{K_\Psi(z , z)} d\nu(w)$$ is in
% $L^p(K_\Psi(z , z)\mu_\Psi(z)).$
%\end{thm}
\begin{proof}[Proof of Theorem~E]
We begin by assuming that $T_\nu$ is in $\cS_p$. Pick a
$\Psi$-lattice $(a_j)$ and let $r$ be its maximal covering radius.
By \eqref{ball-vol} and Lemma~\ref{localball}, we have \[ \aligned
\sum_{k}\left( \frac{\nu( B(a_k, r))}{ |B(a_k, r)|} \right)^p &
\simeq \sum_{k} \left(\int_{ B(a_k, r) } K_{\Psi}(w,w)
d\nu_{\Psi}(w)\right)^p \\ & \simeq \sum_k \left(\int_{B(a_k, r) }
\frac{|K_{\Psi}(a_k,w)|^2}{ K_{\Psi}(a_k,a_k) }
d\nu_{\Psi}(w)\right)^p. \endaligned
\]
By Lemma~\ref{covering2} and our assumption on $\nu$, this gives
\[
\sum_{k}\left( \frac{\nu( B(a_k, r))}{ |B(a_k, r)|} \right)^p
\lesssim  \sum_k \left(\int_{\cn }    \frac{|K_{\Psi}(a_k,w)|^2}{
K_{\Psi}(a_k,a_k) } d\mu_{\Psi}(w) \right)^p. \] If we construct $J$
as in Lemma~\ref{boundedJ}, then the right-hand side equals $
\sum_{k}\left| \langle J^\ast T_\nu J e_k , e_k\rangle \right|^p$.
Since $J$ is a bounded operator, $J^\ast T_\nu J $ also belongs to $
\cS_p$, and so the latter sum converges. We conclude that (i)
implies (ii).

We will use an interpolation argument to prove that (ii) implies
(i). We already know from Theorem~D that $T_\nu$ is in the Schatten
class $\cS_\infty$ whenever $\nu( B(a_k, r))\le C |B(a_k, r)|$ for
some positive constant $C$. Suppose now that
$$ \sum_{k} \frac{\nu( B(a_k, r))}{ |B(a_k, r)|}  < +\infty,$$
and let $(e_j)$ be an orthonormal basis for $\A$. By the reproducing
formula, we have
$$\langle T_\nu e_j , e_j\rangle = \int_{\cn}   |e_j(w)|^2 d\nu_{\Psi}(w),$$
which implies that
\[ \sum_{j} | \langle T_\nu e_j , e_j\rangle |  =
 \int_{\cn}    K_{\Psi}(w,w) d\nu_{\Psi} (w)\le
  \sum_{k} \int_{B(a_k, r) } K_{\Psi}(w,w)d\nu_{\Psi}(w).\]
 Using again Lemma~\ref{lemmah}, we then get
\[ \sum_{j} | \langle T_\nu e_j , e_j\rangle | \lesssim
\sum_{k} \frac{\nu( B(a_k, r))}{ |B(a_k, r)|} < +\infty,\] which
means that $T_\nu$ belongs to $\cS_1.$  By interpolation, we
conclude that (ii) implies (i).
\end{proof}

We remark that the theorems proved in this section generalize
results for the classical Fock space when $n=1$ obtained recently in
\cite{IZ}. It may be noted that Theorem~D above could be elaborated
to include two additional conditions for membership in $\cS_p$, in
accordance with Theorem~4.4 in \cite{IZ}. The proof would be
essentially the same as the proof of the latter theorem. Note that
\cite{IZ} also treats Schatten class membership of Toeplitz
operators for $p<1$.

 \section{Schatten class membership of Hankel operators }

Our work so far suggests two possible definitions of Besov spaces,
in accordance with our respective definitions of $\bmoa(\Psi)$ and
$\M$. We let $\Mmp$ denote the set of entire functions $f$ such that
\[ \int_{\cn} [\mo f(z)]^p K_{\Psi}(z,z)d\mu_{\Psi}(z)<\infty;\]
for a function $h:\cn\to\cn$, we set
\[ |h(z)|_{\beta}=\sup_{\xi \in \cn\setminus\{0\} } \frac{\left
| \langle h(z), \overline{\xi} \rangle\right |}{\beta(z, \xi)},\]
and we let $\Mdp$ be the set of entire functions $f$ for which
\[ \int_{\cn} |\nabla f(z)|_\beta^p K_{\Psi}(z,z)d\mu_{\Psi}(z)<\infty.\]
These definitions are in line with those of K. Zhu for Hankel
operators on the Bergman space of the unit ball in $\cn$ \cite{Zh1}.

It is immediate from \eqref{bmoatobloch} that $\Mmp\subset \Mdp$.
The basic question is whether these spaces coincide and in fact
characterize Schatten class Hankel operators with anti-holomorphic
symbols. The following theorem gives an affirmative answer to this
question.

\newtheorem*{thF}{\bf Theorem F}
\begin{thF}  Let $\Psi$ be a logarithmic growth function, and
suppose that there exists a real number $\eta<1/2$ such that
\eqref{basic} holds and that \eqref{basic2} holds if $n>1$. If $f$
is an entire function on $\cn$ and $p\ge 2$, then the following
statements are equivalent:
\begin{itemize}
\item[(i)]The function $f$ belongs to $\cT(\Psi)$ and the Hankel operator $H_{\bar f} $ on $\A$
is in the Schatten class $\cS_p$;
\item[(ii)] The function $f$ belongs to $\Mmp$;
\item[(iii)] The function $f$ belongs to $\Mdp$.
\end{itemize}
\end{thF}

\begin{proof} We have already observed that the implication (ii)
$\Rightarrow$ (iii) is an immediate consequence of
\eqref{bmoatobloch}. The implication (i) $\Rightarrow$ (ii) relies
on the following general Hilbert space argument. If (i) holds, then
the operator $[H_{\bar f}^{\ast}H_{\bar f}]^{\frac{p}{2}} $ is in
the trace class $ \cS_1.$ Applying Lemma~\ref{traceT} and using the
spectral theorem along with H\"{o}lder's inequality, we obtain
$$    \aligned  \Tr\left([H_{\bar f}^{\ast}H_{\bar f}]^{\frac{p}{2}}\right) &  =  \int_{\cn} \langle [H_{\bar f}^{\ast}H_{\bar f}]^{\frac{p}{2}}  K_{\Psi}(\cdot ,z), K_{\Psi}(\cdot ,z)\rangle d\mu_\Psi(z)
\\
&  \gtrsim  \int_{\cn}\left[\frac{ \|H_{\bar f} K_{\Psi}(\cdot
,z)\|^2 }{K_{\Psi}(z,z) }\right]^{\frac{p}{2}}  K_{\Psi}(z,z)
d\mu_\Psi(z).\endaligned $$ Recalling the computation made in
\eqref{general}, we arrive at (ii).

Our proof of the implication (iii) $\Rightarrow$ (i) will use a
version of L. H\"{o}rmander's $L^2$ estimates for the
$\overline{\partial}$ operator. To this end, write
$\Delta_{\Psi}(z)=\Psi(|z|^2)$ and observe that
\[
\alpha^2(z, \xi):= \sum_{j, k =1}^n\frac{\partial^2
\Delta_{\Psi}(z)} {\partial z_j
\partial \bar z_k}  \xi_j
\bar \xi_k=|\xi|^2 \Psi'(|z|^2)+|\langle z,\xi\rangle|^2
\Psi''(|z|^2)\] for arbitrary vectors $z=(z_1,...,z_n)$ and
$\xi=(\xi_1,...,\xi_n)$ in $\cn$. By Theorem~B, we therefore have
$\alpha(z,\xi)\simeq\beta(z,\xi)$. Now let $L^2_\beta(\mu_\Psi)$ be
the space of vector valued functions $h = (h_1, \cdots h_n)$,
identified with the corresponding $(0,1)$-forms $ h_1 d\bar z_1 +
\cdots + h_nd\bar z_n$ such that
\[ \|h\|_{L^2_\beta(\mu_{\Psi})}^2:=\int_{\cn} |h(z)|_{\beta}^2
d\mu_{\Psi}(z)<\infty. \] It follows from Theorem 2.2 in \cite{BC}
(a special case of a theorem proved by J.-P. Demailly in \cite{D})
that the operator $S$ giving the canonical solution to the
$\bar\partial$-problem is bounded from {\bf $L^2_{\beta}(\mu_\Psi)$
}into $L^2(\mu_\Psi)$.

Since $f$ is holomorphic, we have
$$\bar\partial(H_{\bar f}g) =  \overline{\nabla f}g$$
when $g$ is in $\A$, whence $H_{\bar f}g=S(\overline{\nabla f}g)$.
Thus it follows that
\begin{equation}\label{solution}\|H_{\bar f}g\|_{L^2(\mu_{\Psi})}\lesssim
\int_{\cn} |\nabla f(z)|_\beta^2 |g(z)|^2d\mu_{\Psi}(z).
\end{equation} If we set $d\nu(z)=|\nabla f(z)|_\beta^2dV(z)$, this may be written as
$$H_{\bar f}^{\ast} H_{\bar f}\lesssim M_{|\nabla f|_{\beta}}^\ast
M_{|\nabla f|_\beta} = T_{\nu},$$ where as before $M_h$ denotes the
operator of multiplication by $h$ from $\A$ into $L^2(\mu_{\Psi})$.
By Theorem~E, it remains to verify that (iii) implies that for some
$\Psi$-lattice $(a_k)$ we have
\begin{equation}\label{sufsch}
\sum_{k =1}^{\infty} \left(\frac{\nu( B(a_k, r))}{ |B(a_k, r)|}
\right)^{p/2} < +\infty,\end{equation} where $r$ is the maximal
covering radius for $(a_k)$. To this end, we first observe that
H\"{o}lder's inequality gives that
\[
\left(\frac{\nu( B(z, r))}{ |B(z, r)|} \right)^{p/2}
 \lesssim  \frac{1}{|B(z, r)|}  \int_{B(z, r)} |\nabla f(z)|_{\beta}^p dV(w).\]
Hence, using \eqref{ball-vol} and Lemma~\ref{lemmah}, we obtain
\[\left(\frac{\nu( B(z, r))}{ |B(z, r)|} \right)^{p/2}
 \lesssim  \int_{B(z, r)} |\nabla f(z)|_{\beta}^p K(z,z)dV(w).\]
Now choosing any $\Psi$-lattice $(a_k)$ and using
Lemma~\ref{covering2}, we arrive at
\eqref{sufsch}.
\end{proof}

Several remarks are in order. First note that \eqref{solution} gives
another proof of the implication (iii) $\Rightarrow$ (i) in
Theorem~A, subject to the additional smoothness condition
\eqref{basic2}. Second, as shown in \cite{BY1},  there are
nontrivial Hankel operators in $\cS_p$ only when $p>2n$. This fact
is easy to see from Theorem~F when $n=1$, because then
\[ |\nabla f(z)|_\beta\simeq |f'(z)|[\Phi'(|z|^2)]^{-1/2}, \]
whence $f$ is in $\Mdp$ if and only if
\begin{equation}\label{besov1} \int_{\C} |f'(z)|^p
[\Phi'(|z|^2)]^{1-p/2} dV(z)<\infty.\end{equation} When $n>1$, the
computation of $|\nabla f(z)|_\beta$ is less straightforward, but we
always have
\[ |\nabla f(z)|[\Phi'(|z|^2)]^{-1/2}\lesssim |\nabla f(z)|_\beta
\lesssim  |\nabla f(z)|[\Psi'(|z|^2)]^{-1/2}. \] The estimate from
above shows that the condition \begin{equation}\label{besovn}
\int_{\cn} |\nabla f(z)|^p \Phi'(|z|^2)[\Psi'(|z|^2)]^{n-1-p/2}
dV(z)<\infty\end{equation} is sufficient for $f$ to belong to
$\Mdp$, and the estimate from below shows that this is also
necessary when $\Phi'/\Psi'$ is a bounded function.  We conclude
from \eqref{besov1} and \eqref{besovn} that if the growth of $\Psi'$
is super-polynomial, then $\Mdp$ is infinite-dimensional and
contains all polynomials if and only if $p>2n$. This is immediate
when $n=1$, and it follows also when $n>1$ because
\[ \int_{0}^\infty \frac{\Psi''(t)}{[\Psi'(t)]^{1+\delta}}dt\le
\frac{1}{\delta [\Psi'(0)]^{\delta}}<\infty \] for every $\delta>0$.
 If, on the other hand, $\Psi$
is a polynomial, then $\Phi'/\Psi'$ is a bounded function, and one
may use \eqref{besovn} and Theorem~F to deduce Theorem~B in
\cite{BY2}.

It is not hard to check that if $f$ is a monomial and $n>1$, then
\[|\nabla f(z)|_\beta\simeq |\nabla f(z)||[\Psi'(|z|^2)]^{-1/2}\] for
$z$ belonging to a set of infinite volume measure. By Lemma~2.12 in
\cite{BY2} and Theorem~F above, one may therefore conclude as in
\cite{BY2} that $\Mdp$ is nontrivial only if $p>2n$.

\end{document}